\documentclass[A4]{amsart}

\input xypic
\usepackage{MnSymbol}

\usepackage{bbm}
\usepackage{amsmath}

\newtheorem{theorem}{Theorem}[section]

\newtheorem{proposition}[theorem]{Proposition}

\newcounter{bean}

\newenvironment{romanlist}{\begin{list}{\rm ({\roman{bean}})}
      {\usecounter{bean}\setlength{\rightmargin}{\leftmargin}}}
      {\end{list}}

\newcommand{\seqm}[3]{\ensuremath{#1\stackrel{#2}
 {\longrightarrow}#3}}
\newcommand{\seqmm}[5]{\ensuremath{#1\stackrel{#2}
 {\longrightarrow}#3\stackrel{#4}{\longrightarrow}#5}}
\newcommand{\seqmmm}[7]{\ensuremath{#1\stackrel{#2}
 {\longrightarrow}#3\stackrel{#4}{\longrightarrow}#5
  \stackrel{#6}{\longrightarrow}#7}}
\newcommand{\seqmmmm}[9]{\ensuremath{#1\stackrel{#2}
 {\longrightarrow}#3\stackrel{#4}{\longrightarrow}#5
  \stackrel{#6}{\longrightarrow}#7
  \stackrel{#8}{\longrightarrow}#9}}

\newcommand{\wcolon}{\ensuremath{\,\colon\,}}

\newcommand{\bracket}[1]{\ensuremath{\left( #1 \right)}}
\newcommand{\vbracket}[1]{\ensuremath{\left\langle#1\right\rangle}}

\newcommand{\cplus}[3]{\displaystyle\bigoplus^{#2}_{#1}#3}
\newcommand{\csum}[3]{\displaystyle\sum^{#2}_{#1}#3}

\newcommand{\qqed}{\hfill\square}

\newcommand{\zmodp}{\ensuremath{\mathbb{Z}_{p}}}

\newcommand{\mc}[1]{\ensuremath{\mathcal{#1}}}
\newcommand{\mb}[1]{\ensuremath{\mathbb{#1}}}

\newcommand{\Z}{\ensuremath{\mathbb{Z}}}
\newcommand{\Q}{\ensuremath{\mathbb{Q}}}
\newcommand{\ID}{\ensuremath{\mathbbm{1}}}
\newcommand{\im}{\ensuremath{\mbox{Im }}}

\DeclareMathOperator{\Rank}{rank}

\begin{document}

\title{The Free Loop Space Homology of $(n-1)$-connected $2n$-manifolds}

\author{Piotr Beben}
\address{\scriptsize{School of Mathematics, University of Southampton, Southampton SO17 1BJ, United Kingdom}} 
\email{P.D.Beben@soton.ac.uk} 
\author{Nora Seeliger} 
\address{\scriptsize{Max-Planck-Institut f\"ur Mathematik, Vivatsgasse 7, D-53111 Bonn}}  
\email{seeliger@mpim-bonn.mpg.de} 


\begin{abstract}
Our goal in this paper is to compute the integral free loop space homology of $(n-1)$-connected $2n$-manifolds $M$, $n\geq 2$.
We do this when $n\neq 2,4,8$, or when $n\neq 2$ and $\tilde H^*(M)$ has trivial cup product squares, 
though the techniques used here should extend to a much wider range of manifolds.
We also give partial information concerning the action of the Batalin-Vilkovisky operator.
\end{abstract}

\maketitle
\section{Introduction}

Let $\mc L X=map(S^1,X)$ denote the free loop space on $X$. This space comes equipped with an action
$\nu\colon\seqm{S^1\times\mc L X}{}{\mc L X}$ that rotates loops, and an induced degree $1$ homomorphism 
$$
\Delta\wcolon\seqm{H_*(\mc L X)}{}{H_{*+1}(\mc L X)}
$$
known as the \emph{BV-operator}, defined by setting $\Delta(a)=\nu_*([S^1]\otimes a)$.
In addition Chas and Sullivan~\cite{ChasSullivan} constructed a pairing 
$$
\seqm{H_p(\mc L X)\otimes H_q(\mc L X)}{}{H_{p+q-d}(\mc L X)} 
$$
on a closed oriented $d$-manifold $X$ that (together with the \emph{BV-operator}) turns the shifted homology 
$\mb{H}_*(\mc L X)=H_{*+d}(\mc L X)$ into a Batalin-Vilkovisky (BV)-algebra. 

Batalin-Vilkovisky algebras have been computed in only a few special cases. 
One of the more general results 
to date (due to Felix and Thomas~\cite{FTString}) states that over a field $F$ of characteristic zero and $1$-connected $X$, 
$\mb{H}_*(\mc L X;F)$ is isomorphic to a BV-algebra structure defined on the Hochschild cohomology $HH^*(C^*(X),C^*(X))$. 
Unfortunately, this theorem is generally not true for fields with nonzero characteristic~\cite{Menichi1}. 
Beyond these results, the BV-algebra over various coefficient rings has been completely determined for 
spheres~\cite{CJY,Westerland,Menichi1}, certain Stiefel manifolds~\cite{Tam06}, Lie groups~\cite{Hepworth1}, 
and projective spaces~\cite{Ziller,Seeliger,CJY,Tian,Hepworth2}, 
using a mixture of techniques ranging from homotopy theoretic to geometric, as well as the well-known connections to Hochschild cohomology. 

In this paper we focus on the free loop space homology of highly connected $2n$-manifolds, together with 
the action of the BV-operator. The coefficient ring $R$ for homology and cohomology is assumed to be either 
any field, or the integers $\Z$, but we suppress it from notation most of the time.  
Fix $n\geq 2$, $M$ a $(n-1)$-connected, closed, oriented $2n$-manifold with $H^{n}(M)$ of rank $m\geq 1$.
Let 
$$
C=[c_{ij}=\vbracket{a_i\cup a_j,[M]}]
$$ 
be the $m\times m$ matrix for the intersection form \seqm{H^{n}(M)\times H^{n}(M)}{}{\mb Z} with respect to some 
choice of basis $\{a_1,\ldots,a_m\}$ for $H^{n}(M)$ (we use the same notation for the dual basis of $H^{n}(M)$). 
This form is nonsingular, symmetric when $n$ is even, and skew-symmetric when $n$ is odd.

Denote $H^{n}(M)$ and $H^{2n}(M)\cong\mb Z$ by the free graded modules $R$-modules $A=R\{a_1,\ldots,a_m\}$ and $K=R\{[M]\}$, 
and the desuspension of $A$ by $V=R\{u_1,\ldots,u_m\}$ with $|u_i|=n-1$. Let
$$
T(V)=R\oplus \cplus{i\geq 1}{}{V^{\otimes i}}
$$
be the free tensor algebra generated by $V$, and $I$ be the two-sided ideal of the tensor algebra $T(V)$
generated by the following degree $2n-2$ element 
$$
\chi=\csum{i<j}{}{c_{ij}[u_i,u_j]}+\csum{i}{}{c_{ii}u_i^2},
$$
where $[x,y]=xy-(-1)^{|x||y|}yx$ denotes the graded Lie bracket in $T(V)$. Take the quotient algebra
$$
U=\frac{T(V)}{I}
$$ 
and the degree $-1$ maps of graded $R$-modules $d\colon\seqm{A\otimes U}{}{U}$ and 
$d'\colon\seqm{K\otimes U}{}{A\otimes U}$, which are given for any $y\in U$ by the formulas
$$
d(a_i\otimes y)= [u_i,y]
$$
$$
d'([M]\otimes y)=\csum{i,j}{}{c_{ij} (a_j\otimes [u_i,y])}.
$$

If we apply the Jacobi identity to the summands $c_{ij}(a_j\otimes [u_i,y])$ in $d\circ d'(y)$ for $i<j$ 
(keeping in mind that $c_{ij}=(-1)^{n}c_{ji}$, $[u_i,[u_i,y]]=[u_i^2,y]$, and that products with $\chi$ 
are identified with zero in $U$), we see that $\im d'\subseteq \ker d$, so we obtain a chain complex
$$
\seqmmmm{0}{}{K\otimes U}{d'}{A\otimes U}{d}{U}{}{0}. 
$$
Now take the homology of this chain complex. That is, take the following graded $R$-modules:
$$
\mc Q=\frac{U}{\im d},\quad
\mc W=\frac{\ker d}{\im d'},\quad
\mc Z=\ker d'.
$$
One can think of $\mc W$ by first taking the $R$-submodule $W'$ of $\Sigma^{-1} A\otimes T(V)\cong T(V)$ 
generated by elements that are invariant modulo $I$ under graded cyclic permutations, that is, invariant 
after projecting to $U$. Then $\mc W$ is the projection of $\Sigma W'$ onto $(A\otimes U)/\im d'$.

Our main result is that the homology of this chain complex is the integral free loop space homology of $M$ under some conditions:

\begin{theorem}
\label{TMain1}
Suppose $n\geq 2$ and $\Rank H^{n}(M)\geq 1$. 
If $n\neq 2,4,8$, or $n\neq 2$ and $\tilde H^*(M)$ has trivial cup product squares, 
then there exists an isomorphism of graded $R$-modules
$$
H_*(\mc L M)\cong \mc Q\oplus \mc W\oplus \mc Z.
$$  
\end{theorem}

The restriction away from $2$,$4$, and $8$ traces back to an argument that we use to determine $H_*(\Omega M)$,
which does not apply to situation where there are cup product squares equal to the fundamental class $[M]$, or $-[M]$.
Failure of a degree placement argument to compute certain differentials is another reason that we restrict 
away from $n=2$.

We also determine the action of the BV-operator on $H_*(\mc L M;\Q)$, in a sense, up-to-abelianization of $U$ when $n>3$ is 
odd. 

Consider the graded abelianization map \seqm{T(V)}{\eta}{S(V)}, where $S(V)$ is the free graded symmetric algebra 
generated by $V$. Since $\eta(\chi)=0$, $\eta$ factors through \seqm{U}{\eta}{S(V)}. Also, consider the maps 
\seqm{A\otimes U}{\ID_A\otimes\eta}{A\otimes S(V)} and \seqm{K\otimes U}{\ID_K\otimes\eta}{K\otimes S(V)}. Since 
$(\ID_A\otimes \eta)\circ d'=0$ and $\eta\circ d=0$, then $\eta$ and these two maps induce \emph{abelianization} maps 
$$
\seqm{\mc Q}{\eta_q}{S(V)},
$$
$$
\seqm{\mc W}{\eta_w}{A\otimes S(V)},
$$
$$
\seqm{\mc Z}{\eta_z}{K\otimes S(V)}.
$$
\begin{theorem}
\label{TMain2}
Let $n>3$ be odd. The BV operator $\Delta\colon\seqm{H_*(\mc L M;\Q)}{}{H_{*+1}(\mc L M;\Q)}$ satisfies 
$\Delta(\mc Q)\subseteq \mc W$ and $\Delta(\mc W)\subseteq \mc Z$, and $\Delta(\mc Z)=\{0\}$. Moreover,
the composite \seqmm{\mc Q}{\Delta}{\mc W}{\eta_w}{A\otimes S(V)} is given by
$$
\eta_w\circ\Delta(1\otimes(u_{i_1}\cdots u_{i_k}))=
\csum{j=1}{k}{a_{i_j}\otimes(u_{i_1}\cdots u_{i_j-1}u_{i_j+1}\cdots u_{i_k})},
$$
and \seqmm{\mc W}{\Delta}{\mc Z}{\eta_z}{K\otimes S(V)} is the restriction to $\ker d$ 
of the map \seqm{(A\otimes U)/\im d'}{}{S(A)\otimes S(V)} given by
$$
a_i\otimes(u_{i_1}\cdots u_{i_k})\mapsto
\csum{j=1}{k}{a_i a_{i_j}\otimes(u_{i_1}\cdots u_{i_j-1}u_{i_j+1}\cdots u_{i_k})},
$$
where $[M]\in K$ is identified with $(\sum_{i\leq j}c_{ij}a_ia_j)\in S(A)$.
\end{theorem}

Berglund and B\"orjeson~\cite{BerglundBorjeson} have subsequently computed the free loop space homology of highly connected manifolds 
(including the ones considered in this paper) using different techniques.
They also give a description of the action of the BV-operator and the Chas-Sullivan loop product.
With a bit of effort it is likely that the spectral sequence methods in this paper can be extended to cover many of the highly connected manifolds in~\cite{BerglundBorjeson}.
For example, the based loop space homology of highly connected manifolds is largely known~\cite{MR3391363}, and this is one of the main ingredients used in our calculations. 
On the other hand, we do not know whether a complete description of the Chas-Sullivan loop product and BV-operator is possible using our approach -
one difficulty being extension issues in the Cohen-Jones-Yan spectral sequence~\cite{MR2039760} when computing the loop product,
together with a seeming incompatibility between the BV-operator and the Serre spectral sequence of a free loop fibration.

We should mention that there are sources of application for the above calculations that go beyond the classical question:
\emph{are there infinitely many geometrically distinct periodic geodesics on a Riemannian manifold $M$?}
For example, detailed information about the Betti numbers of $\mc L M$ reflects more detailed information about the number of geodesics of variable length.
See~\cite{BasuBasu,BerglundBorjeson,MR672499,MR551562}.

\section{A Useful Lemma}

Take a fibration sequence $\seqmm{F}{i}{X}{f}{B}$ with $B$ simply-connected. Recall the induced homotopy fibration sequence
\begin{equation}
\label{EPrincipalFib}
\seqmm{\Omega B}{\vartheta}{F}{i}{X}
\end{equation}
is a \emph{principal} homotopy fibration. 
Namely, there is a homotopy associative $H$-space structure on the homotopy fiber $\Omega B$ together with a left action
$$
\theta\wcolon\seqm{\Omega B\times F}{}{F}
$$ 
that fits into a homotopy commutative square
\[\diagram
\Omega B\times \Omega B\rto^{\ID\times \vartheta}\dto_{mult.} & \Omega B\times F\dto^{\theta}\\
\Omega B\rto^{\vartheta} & F.
\enddiagram\]
In our case the $H$-space multiplication $mult.$ on $\Omega B$ is taken as the one defined by composing loops, 
and the action $\theta$ is defined by applying the homotopy lifting property to loops in $B$. 

By a result of Moore~\cite{Moore}, the homology Serre spectral sequence $\xi$ of a principal fibration such as~(\ref{EPrincipalFib}) 
has a left $H_*(\Omega B)$-module induced by the associated action $\theta$.
Namely, there is a left action $\seqm{H_*(\Omega B)\otimes \xi^r_{i,j}}{}{\xi^r_{i,j+*}}$ reducing to the Pontrjagin multiplication on 
$\xi^2_{0,*}\cong H_*(\Omega B)$ and differentials respect this action.
Most of the effort in computing differentials is therefore reduced to determining those emanating from the degree $0$ horizontal line. 

Since fibrations are characterized by the homotopy lifting property, 
one might also expect $\theta$ to have a direct bearing on the homology Serre spectral sequence for our original fibration $f$. 
This was exploited by McCleary in~\cite{McCleary}, where he used a result of Brown~\cite{Brown} and Shih~\cite{Shih} 
to give a computation of the free loop space homology of certain low rank Stiefel manifolds. 
The following proposition strengthens the result in~\cite{Brown,Shih} by doing away with an assumption about certain elements being trangressive. 
The proof is moreover fairly simple. Let 
$$
\mc E=\{\mc E^r,\delta^r\}
$$ 
denote the homology Serre spectral sequence for $f$, and
$$
E=\{E^r,d^r\}
$$ 
the homology Serre spectral sequence for the path-loop fibration 
sequence $\seqmm{\Omega B}{\subset}{\mc P B}{ev_1}{B}$.

\begin{proposition}
\label{PAction}
Suppose $H_*(B)$ and $H_*(F)$ are torsion free.
Given $z\in H_*(B)$, and $\sum_i x_i\otimes v_i\in E^2_{*,*}\cong H_*(B)\otimes H_*(\Omega B)$,
suppose $d^s(z\otimes 1)=d^s(\sum_i x_i\otimes v_i)=0$ in $E^s_{*,*}$ for $2\leq s<r$, 
and
$$
d^r(z\otimes 1)=\sum_i x_i\otimes v_i.
$$ 
Then given $z\otimes y\in\mc E^2_{*,*}\cong H_*(B)\otimes H_*(F)$ for any $y\in H_*(F)$, 
for each $2\leq s<r$ we have 
$$
\delta^s(z\otimes y)=\delta^s(\sum_i x_i\otimes \theta_*(v_i\otimes y))=0
$$
and
$$
\delta^r(z\otimes y)= \sum_i x_i\otimes \theta_*(v_i\otimes y).
$$ 
\end{proposition}

\begin{proof}
First recall the following well-known property (which is essentially the homotopy lifting property in disguise).
Let $P^{ev_0,f}\subseteq map([0,1],B)\times X$ be the pullback of \seqm{X}{f}{B} and the evaluation map \seqm{map([0,1],B)}{ev_0}{B}, 
where $ev_t(\omega)=\omega(t)$. 
Now consider the map $\bar f\colon\seqm{map([0,1],X)}{}{P^{ev_0,f}}$ defined by $\bar f(\omega)=(f\circ\omega,\omega(0))$.
Then a surjection $f$ is a fibration if and only if there exists a map $g\colon\seqm{P^{ev_0,f}}{}{map([0,1],X)}$
such that $\bar f\circ g=\ID\colon\seqm{P^{ev_0,f}}{}{P^{ev_0,f}}$.

Take the inclusion 
$
\phi\wcolon\seqm{\mc P B \times F}{}{P^{ev_0,f}} 
$
given by $\phi(\omega,a)=(\omega,a)$, and take the the composite
$$
\bar\theta\wcolon\seqmmm{(\mc P B\times F)}{\phi}{P^{ev_0,f}}{g}{map([0,1],X)}{ev_1}{X}.
$$
Let the fibration sequence
\begin{equation}
\label{Eprodfib}
\seqmm{\Omega B\times F}{\subset\times\ID}{\mc P B\times F}{ev_1\times\ast}{B\times\ast}
\end{equation}
be the product of the path-loop fibration sequence \seqmm{\Omega B}{\subset}{\mc P B}{ev_1}{B} and the trivial 
fibration sequence \seqmm{F}{\ID}{F}{\ast}{\ast}.
Let $E=\{E^s,d^s\}$ and $\mathring{E}=\{\mathring{E}^s,\mathring{d}^s\}$ be the homology Serre spectral sequences 
for the path-loop and trivial fibration respectively,
and $\hat E=\{\hat E^s,\hat d^s\}$ be the homology spectral sequence for their product~(\ref{Eprodfib}).
Define a differential $d^s_{\otimes}\colon\seqm{E^s\otimes \mathring{E}^s}{}{E^s\otimes \mathring{E}^s}$ by 
$\hat d^s(a\otimes b)=(d^s(a)\otimes b)+(-1)^{|a|}(a\otimes\mathring{d}^s(b))$. 
Since $H_*(F)$ is torsion-free, $\hat E^s=E^s\otimes \mathring{E}^s$ and $\hat d^s=d^s_{\otimes}$ (see~\cite{MR0145530,MR0086301}). 
In our case $\mathring{d}=0$, so we have
$$
\hat d^s(a\otimes b)=d^s(a)\otimes b
$$ 
for any $a\in E^s$ and $b\in \mathring{E}^s$.
One can easily check that the following diagram of fibration sequences commutes
\begin{equation}
\label{ELast}
\diagram
\Omega B\times F\rto^{\subset\times\ID}\dto^{\theta} & \mc P B\times F\rto^-{ev_1\times\ast}\dto^{\bar\theta} & B\times\ast\ddouble\\
F\rto^{i} & X\rto^-{f} & B,
\enddiagram
\end{equation}
with our action $\theta$ being in fact the restriction of $\bar\theta$ to the subspace $\Omega B\times F$.
Let
$$ 
\zeta\wcolon\seqm{\hat E=E\otimes\mathring{E}}{}{\mc E} 
$$
be the morphism of spectral sequences induced by this diagram. 

Since $d^s(z\otimes 1)=0\in E^s_{*,*}$ for $2\leq s<r$ and $d^r(z\otimes 1)=\sum_i x_i\otimes v_i$, 
then for any $b\in \mathring{E}^s$
$$
\hat d^s((z\otimes 1)\otimes b)=d^s(z\otimes 1)\otimes b=0
$$
$$
\hat d^r((z\otimes 1)\otimes b)=d^r(z\otimes 1)\otimes b=\sum_i (x_i\otimes v_i)\otimes b,
$$
which we use to obtain
\begin{align*}
\delta^r(z\otimes y)&=\delta^r(\,\zeta^r((z\otimes 1)\otimes (1\otimes y))\,)\\
&=\zeta^r(\,\hat d^r((z\otimes 1)\otimes (1\otimes y))\,)\\
&=\zeta^r\bracket{\sum_i (x_i\otimes v_i)\otimes (1\otimes y)}\\
&=\sum_i x_i\otimes \theta_*(v_i\otimes y),
\end{align*}
and similarly, $\delta^s(z\otimes y)=0$ for $2\leq s<r$.

In a similarly manner, 
we see $\hat d^s((\sum_i x_i\otimes v_i)\otimes b)=0$ for $2\leq s<r$ and (in turn) $\delta^s(\sum_i x_i\otimes\theta_*(v_i\otimes y))=0$
using the fact that $d^s(\sum_i x_i\otimes v_i)=0$ (so the above equations make sense).

\end{proof}

We now turn our attention towards the free loop space fibration sequence
\begin{equation}
\label{Eloopfib}
\seqmm{\Omega B}{\vartheta}{\mc L B}{ev_1}{B}.
\end{equation}
The map $\vartheta$ is the canonical inclusion $\Omega B\subseteq \mc L B$, and $ev_1$ is the evaluation map
$ev_1(\omega)=\omega(1)$. The homology Serre spectral sequence for this fibration sequence will be denoted by
$$
\mc E=\{\mc E^r,\delta^r\},
$$
and as before $E=\{E^r,d^r\}$ is the homology Serre spectral sequence for the path-loop fibration of $B$.
The path-loop fibration is principal, 
so $E$ has a left $H_*(\Omega B)$-module as described before which the differentials $d$ respect.

Some basic properties of the free loop space fibration are as follows. The map \seqm{\mc L B}{ev_1}{B} has a section 
\seqm{B}{s}{\mc L B} defined by mapping a point $b\in B$ to the constant loop at $b$, which implies the connecting 
map $\varrho$ for the induced principal homotopy fibration \seqmm{\Omega B}{\varrho}{\Omega B}{\vartheta}{\mc L B} 
is null homotopic. The associated left action 
$$
\theta\wcolon\seqm{\Omega B\times\Omega B}{}{\Omega B}
$$
is given by 
$$
\theta(\omega,\lambda)=\omega\cdot\lambda\cdot\omega^{-1}
$$
for any $\omega,\lambda\in \Omega B$. 
If $v\in H_*(\Omega B)$ is primitive, then for any $y\in H_*(\Omega B)$ one has the formula
$$
\theta_*(v\otimes y)=(-1)^{|v||y|}yv-vy=-[v,y],
$$
where the multiplication on $H_*(\Omega B)$ is the Pontrjagin multiplication induced by loop composition on $\Omega B$. 
The proof of these can be found in~\cite{McCleary} for example.
Combining these properties with Propositions~\ref{PAction} gives the following description of the differentials 
in the spectral sequence $\mc E$.

\begin{proposition}
\label{PDiff}
Suppose $H_*(B)$ and $H_*(\Omega B)$ are torsion free, and $B$ is $1$-connected.
Given $z\in H_*(B)$, and $\sum_i x_i\otimes v_i\in E^2_{*,*}$ with $v_i$ primitive in $H_*(\Omega B)$,
suppose that $d^s(z\otimes 1)=0$ and $d^s(\sum_i x_i\otimes v_i)=0$ in $E^s_{*,*}$ for $2\leq s<r$, 
and
$$
d^r(z\otimes 1)=\sum_i x_i\otimes v_i.
$$ 
Then given $z\otimes y\in\mc E^2_{*,*}$ for any $y\in H_*(\Omega B)$, 
for each $2\leq s<r$ we have 
$$
\delta^s(z\otimes y)=\delta^s(\sum_i x_i\otimes [v_i,y])=0
$$
and
$$
\delta^r(z\otimes y)= -\sum_i x_i\otimes [v_i,y].
$$~$\qqed$
\end{proposition}

There are instances where this formula fails to give us enough information to determine some of the higher differentials. 
For example, if we found ourselves in the situation where $\delta^s(z\otimes y)=0$ for $s\leq r$ and $d^r(z\otimes y)\neq 0$, 
then $z\otimes y\in\mc E^{r}_{*,*}$ survives to the $\mc E^{r+1}$ page, while $z\otimes y$ is not an element in $E^{r+1}_{*,*}$. 
In such case $\delta^s(z\otimes y)$ remains mysterious when $s>r$. 
An example where this situation happens in practice is the case of $4$-manifolds omitted from Theorem~\ref{TMain1}.

\section{Based Loop Space Homology}

Returning to our $2n$-manifold $M$ in the introduction, we consider the Hopf algebra $H_*(\Omega M)$. 
This is the last piece in the puzzle required to prove Theorem~\ref{TMain1}.
By Poincar\'e duality the only nonzero reduced homology groups of $M$ are in degrees $n$ and $2n$. 
This implies $M$ has a cell decomposition given by attaching an $n$-cell to an $m$-fold wedge of $n$-spheres
$\bigvee_m S^{n}\simeq M-*$, where $m=\Rank H^{n}(M)$.

Generally, if a space $Y$ is formed by attaching a $k$-cell to a space $X$ via an attaching map
\seqm{S^{k-1}}{\alpha}{X}, and $\alpha'$ is its adjoint, the composite with the looped inclusion 
\seqmm{S^{k-2}}{\alpha'}{\Omega X}{\Omega i}{\Omega Y} is nullhomotopic, so one obtains a factorization 
of Hopf algebras through Hopf algebra maps
\begin{equation}
\label{DFactor}
\diagram
& H_*(\Omega X;R)/I\dto^{\theta}\\
H_*(\Omega X;R)\rto^{(\Omega i)_*}\urto^{} & H_*(\Omega Y;R), 
\enddiagram
\end{equation}
where $I$ is the two-sided ideal generated by $\alpha'([S^{k-2}])\in H_{k-2}(\Omega X;R)$. The problem 
of determining the conditions under which $\theta$ is a Hopf algebra isomorphism is part of what is known 
as the \emph{cell-attachment problem}. One of these conditions - the \emph{inert} condition - states somewhat suprisingly 
that $\theta$ is a Hopf algebra isomorphism when $R$ is a field if and only if $(\Omega i)_*$ is a surjection
~(\cite{Lemaire,HalperinLemaire,FelixThomas}). Here we select $k=2n$, $Y\simeq M$, and $X\simeq M-*$, and use 
the inert condition to prove the following:

\begin{proposition}
\label{P1}
Suppose $n\geq 2$ and $\Rank H^{n}(M)\geq 1$. 
If $n\neq 2,4,8$, or $\tilde H^*(M)$ has trivial cup product squares, 

\begin{romanlist}
\item There is an isomorphism of Hopf algebras (free as $R$-modules) 
$$
H_*(\Omega M)\cong \frac{T(V)}{I}
$$
where $V=R\{u_1,\ldots,u_m\}$, $|u_i|=n-1$.
\item The element $\alpha'_*([S^{2n-2}])$ generating the two-sided ideal $I$
is given by
$$
\alpha'_*([S^{2n-2}])=\csum{i<j}{}{c_{ij}[u_j,u_i]}+\csum{i}{}{c_{ii}u_i^2}.
$$
\end{romanlist}
\end{proposition}

\begin{proof}[Proof of part (i)]

In~\cite{MR3228428}, $\Omega M$ is shown to be a homotopy retract of $\Omega(M-*)$ when $n\neq 2,4,8$,
or when $\tilde H^*(M)$ has trivial cup product squares.
Therefore $(\Omega i)_*$ is a split epimorphism, so we obtain
$H_*(\Omega M;F)\cong H_*(\Omega(M-*);F)/I$ for any field $F$. 
Moreover, since $M-*$ is homotopy equivalent to $\bigvee_m S^{n}$, the $\Z$-module 
$H_*(\Omega(M-*);\Z)\cong T(V)$ is torsion-free. Therefore $H_*(\Omega M;\Z)$ is torsion-free, 
and the Hopf algebra isomorphism holds for $R=\Z$ as well. 

\end{proof}

\begin{proof}[Proof of part (ii)]

We will write $u_j=(\Omega i)_*(u_j)\in H_{n-1}(\Omega M)$, and take $u_j$ to be the transgression of 
$a_j\in H_{n}(M)$.

Since the elements $u_1,\ldots,u_m$ in $H_{n-1}(\Omega(M-*))$ are primitive, and there are no monomials 
of length greater than $2$ in degree $2n-2$, the elements $u_i^2$ and $[u_j,u_i]$
form a basis for the primitives in $H_{2n-2}(\Omega(M-*))$. Now $\alpha'_*([S^{2n-2}])$ is primitive since 
$[S^{2n-2}]$ is primitive, so we can set
$$
(\alpha')_*([S^{2n-2}])=\csum{i<j}{}{c''_{ij}[u_i,u_j]}+\csum{i}{}{c''_{ii}u_i^2}
$$
for some integers $c''_{ij}$.

Consider the homology Serre spectral sequence $E=(E^r,d^r)$ for the (principal) path-loop fibration sequence $M$, 
with
$$
E^2_{*,*}=H_*(M)\otimes H_*(\Omega M).
$$
On the dual cohomology spectral sequence we have the formula
$$
d_{n}(a_j\otimes u_i)=d_{n}(a_j\otimes 1)(1\otimes u_i)+(-1)^n(a_j\otimes 1)d_{n}(1\otimes u_i)
=(-1)^n(a_j\otimes 1)(a_i\otimes 1)= c_{ij}([M]^*\otimes 1),
$$
so dualizing back to the homology spectral sequence gives us 
\begin{align}
\label{EDiff}
d^{n}([M]\otimes 1)=\csum{i,j}{}{c_{ij} (a_j\otimes u_i)}.
\end{align}

Take $\bar E=(\bar E^r,\bar d^r)$ to be the homology Serre spectral sequence for the path-loop fibration of $M-*$.
The inclusion \seqm{(M-*)}{}{M} induces an inclusion of the corresponding path-loop fibrations of $(M-*)$ and $M$, 
and in turn a morphism of spectral sequences
$
\gamma\wcolon\seqm{\bar E}{}{E}.
$
On the second page of spectral sequences $\gamma_2$ maps 
$1\otimes u_i$ to $1\otimes u_i$ and $a_i\otimes 1$ to $a_i\otimes 1$, and 
\seqm{\bar E^{r}_{n,n-1}}{\gamma_r}{E^{r}_{n,n-1}} is an isomorphism for $2\leq r\leq n$. 

By part (i) of the theorem (and preceeding discussion), $(\alpha')_*([S^{2n-2}])$ generates the kernel 
of $(\Omega i)_{*}\colon\seqm{H_{2n-2}(\Omega(M-*))}{}{H_{2n-2}(\Omega M)}$, 
so $1\otimes(\alpha)_*([S^{2n-2}])$ generates the kernel of $\gamma_2\colon\seqm{E^{2}_{0,2n-2}}{}{E^{2}_{0,2n-2}}$. 
Since $\gamma_r\colon\seqm{\bar E^{r}_{i,j}}{}{E^{r}_{i,j}}$ is an isomorphism for $i<n$, $j<2n-2$, and all $r$, 
then in fact $1\otimes(\alpha')_*([S^{2n-2}])$ generates the kernel of the map 
\seqm{\bar E^{r}_{0,2n-2}}{\gamma_r}{E^{r}_{0,2n-2}} for $2\leq r\leq n$.

Take the element
$$
\zeta''=\csum{i\leq j}{}{c''_{ij} (a_j\otimes u_i-a_i\otimes u_j)}
$$
in $\bar E^{r}_{n,n-1}$, for $2\leq r\leq n$. Then 
\begin{equation}
\label{EGammaZeta}
\gamma_{n}(\zeta'')=\csum{i\leq j}{}{c''_{ij} (a_j\otimes u_i-a_i\otimes u_j)},
\end{equation}
and in $\bar E^{n}_{0,2n-2}$ we have 
$$
1\otimes(\alpha')_*([S^{2n-2}])=\csum{i\leq j}{}{c''_{ij}(1\otimes[u_i,u_j])}=\bar d^{n}(\zeta'').
$$
Since $\bar E^{r}_{i,j}=\{0\}$ for $i>n$ and $\bar E^{\infty}_{\ast,\ast}=\{0\}$, 
the differential \seqm{\bar E^{n}_{n,n-1}}{\bar d^{n}}{\bar E^{n}_{0,2n-2}} is an isomorphism,
and since \seqm{\bar E^{n}_{n,n-1}}{\gamma_{n}}{E^{n}_{n,n-1}} is an isomorphism
and $1\otimes(\alpha')_*([S^{2n-2}])$ generates the kernel of \seqm{\bar E^{n}_{0,2n-2}}{\gamma_{n}}{E^{n}_{0,2n-2}},
by naturality we see that the kernel of the differential \seqm{E^{n}_{n,n-1}}{d^{n}}{E^{n}_{0,2n-2}} is generated by $\gamma_{n}(\zeta'')$. 
In particular, we may project $\gamma_{n}(\zeta'')$ down to $E^{\infty}_{*,*}$. 

Let
$$
\mc I=\im d^{n}\wcolon\seqm{E^{n}_{2n,0}}{}{E^{n}_{n,n-1}}
$$
$$
\mc K=\ker d^{n}\wcolon\seqm{E^{n}_{n,n-1}}{}{E^{n}_{0,2n-2}}.
$$
As we saw above, $\mc I$ is generated by $d^{n}([M]\otimes 1)$, and $\gamma_{n}(\zeta'')$ generates $\mc K$. 
But the short exact sequence
$$
\seqmmmm{0}{}{E^{n}_{2n,0}}{d^{n}}{E^{n}_{n,n-1}}{d^{n}}{E^{n}_{0,n-2}}{}{0}
$$
implies $\mc I\subseteq \mc K$. Therefore $d^{n}([M]\otimes 1)=\pm\gamma_{n}(\zeta'')$. 
Now comparing coefficients in equations~(\ref{EDiff}) and~(\ref{EGammaZeta}), the result follows.

\end{proof}

\section{Proof of Theorem~\ref{TMain1}}

We now have everything required to prove Theorem~\ref{TMain1} via a routine Serre spetral sequence argument.
Let $\mc E=\{\mc E^r,\delta^r\}$ be the homology Serre spectral sequence for the free loop space fibration 
sequence 
$$
\seqmm{\Omega M}{\vartheta}{\mc L M}{ev_1}{M}.
$$
By Proposition~\ref{P1} we have an isomorphism $H_*(\Omega M)\cong U=T(V)/I$ of Hopf algebras, which are 
free as $R$-modules. So we start with an isomorphism of free $R$-modules 
$$
\mc E^2_{*,*}\cong R\{1,a_1,\ldots,a_m,[M]\}\otimes U.
$$
By Proposition~\ref{PDiff}
$$
\delta^{n}(a_i\otimes y)= -1\otimes [u_i,y]
$$
where $u_i$ is the transgression of $a_i$, and using~(\ref{EDiff}),
$$
\delta^{n}([M]\otimes y)=-\csum{i,j}{}{c_{ij} (a_j\otimes [u_i,y])}.
$$
Therefore $\mc E^{2n}_{0,*}\cong \mc Q$, $\mc E^{\infty}_{n,*}\cong \mc E^{2n}_{n,*}\cong\mc W$, and 
$\mc E^{2n}_{2n,*}\cong \mc Z$, while all other entries in the spectral sequence are zero. 
Here, the only possible nonzero differentials are $\delta^{2n}\colon\seqm{\mc E^{2n}_{2n,*}}{}{\mc E^{2n}_{0,*+2n-1}}$.
But since the nonzero elements in $\mc Z$ and $\mc Q$ are concentrated in total degrees $2n+k(n-1)$ and $k(n-1)$ respectively, 
one can check the differentials $\delta^{2n}$ are zero for degree placement reasons whenever $n>2$. 
Thus these isomorphisms carry over to the infinity page, that is, 
$$
\mc E^{\infty}_{*,*}\cong\mc E^{\infty}_{0,*}\oplus\mc E^{\infty}_{n,*}\oplus\mc E^{\infty}_{2n,*}\cong\mc Q\oplus\mc W\oplus\mc Z. 
$$

Generally, one has torsion here when $R=\Z$ (or at least in $\mc Q$, and possibly $\mc W$), so we must consider 
a potential extension problem. Once again placement reasons allow us to skirt around the issue.

From the construction of the homology Serre spectral sequence there are increasing filtrations 
$\mc F_{i,j}=\mc F_i H_j(\mc L M)\subseteq H_j(\mc L M)$ such that $\mc F_{k,k}=H_k(\mc L M)$, $\mc F_{i,j}=0$ 
for $i<0$, and 
$$
\mc E^{\infty}_{i,j}\cong\frac{\mc F_{i,i+j}}{\mc F_{i-1,i+j}}.
$$
Since the nonzero elements in $\mc Q$, $\mc W$, and $\mc Z$ are in degrees $k(n-1)$, $n+k(n-1)$, and $2n+k(n-1)$, 
$\mc Q$, $\mc W$, and $\mc Z$ pairwise have no nonzero elements in the same degrees when $n>3$. Since 
$\mc F_{n-1,*}=\mc F_{0,*}=\mc Q$, we have $\mc F_{n-1,n+k(n-1)}=\{0\}$, and we see that
$\mc F_{n,*}\cong\mc F_{0,*}\oplus\mc E^{\infty}_{n,*}\cong \mc Q\oplus\mc W$. Then 
$\mc F_{2n-1,2n+k(n-1)}=\mc F_{n,2n+k(n-1)}=\{0\}$, so $\mc F_{2n,*}\cong\mc F_{n,*}\oplus\mc E^{\infty}_{2n,*}$, 
and we have
$$
\mc E^{\infty}_{2n,*}\cong\mc F_{2n,*}=H_*(\mc L M)
$$  
whenever $n>3$. 

When $n=3$, the common nonzero degrees shared between any pair of these three modules are of the form $2(k+3)$, 
and these are only between $\mc Q$ and $\mc Z$. But since $\mc Z$ is torsion-free and $\mc Q=\mc F_{0,*}$ is at
the bottom of the filtration, there are no extension issues here either.

\section{Eilenberg-Maclane Spaces and the BV-operator}

We will need some information about the action of the BV-operator on products of Eilenberg-Maclane spaces before getting into the proof 
Theorem~\ref{TMain2}. The approach we take here is similar to the one taken by Hepworth in~\cite{Hepworth1} to compute the 
BV-operator for Lie groups. We begin this section by recalling it. Fix $R$ to be a principal ideal domain, and $X$ 
(homotopy type of a $CW$-complex) a path-connected topological group with multiplication \seqm{X\times X}{\mu}{X}. 
This makes $\mc L X$ into topological group with multiplication \seqm{\mc L X\times\mc L X}{\mc L\mu}{\mc L X} defined 
by point-wise multiplication of loops $(\omega\cdot\gamma)(t)=\omega(t)\cdot\gamma(t)$. There is a well-known homeomorphism 
\begin{align*}
h\wcolon\seqm{X\times\Omega X}{}{\mc L X}\\
h(x,\omega)=x\cdot \omega
\end{align*}
with inverse $h^{-1}\colon\seqm{\mc L X}{}{X\times\Omega X}$ given by $h^{-1}(\omega)=(\omega(0),\omega(0)^{-1}\cdot\omega)$, 
where $x\cdot\omega$ is the loop defined at each point by $(x\cdot\omega)(t)=x\cdot\omega(t)$.
These homeomorphisms are equivariant with respect to our action \seqm{S^1\times\mc L X}{\nu}{\mc L X}, and the action
$$
\bar\nu\wcolon\seqm{S^1\times X\times\Omega X}{}{X\times\Omega X}
$$
defined by the formula 
$$
\bar\nu(t,x,\omega)=h^{-1}\circ\nu(t,x\cdot\omega)=\bracket{x\cdot\omega_t(0),(x\cdot\omega_t(0))^{-1}\cdot x\cdot\omega_t}
=\bracket{x\cdot\omega_t(0),\omega_t(0)^{-1}\cdot\omega_t}
$$
where $\omega_t(s)=\nu(t,\omega)(s)=\omega(s+t)$. On homology we have a commutative square
\[\diagram
H_*(X\times\Omega X;R)\dto_{\bar\Delta}\rto^{h_*}_{\cong} & H_{*}(\mc L X;R)\dto^{\Delta}\\
H_{*+1}(X\times\Omega X;R)\rto_{h_*}^{\cong} & H_{*+1}(\mc L X;R).
\enddiagram\]
where $\bar\Delta(e)=\bar\nu_*([S^1]\otimes e)$. Clearly, after transposing $X$ and $S^1$, $\bar\nu$ is the composite 
$$
\seqmmm{X\times(S^1\times\Omega X)}{\ID_X\times\vartriangle}{X\times (S^1\times\Omega X)\times(S^1\times\Omega X)}
{\ID_X\times ev\times\phi}{(X\times X)\times\Omega X}{\mu\times\ID}{X\times\Omega X},
$$
with $ev\colon\seqm{S^1\times\Omega X}{}{X}$ the evaluation map $ev(t,\omega)=\omega(t)=\omega_t(0)$, and
$\phi\colon\seqm{S^1\times\Omega X}{}{\Omega X}$ defined by $\phi(t,\omega)=\omega_t(0)^{-1}\cdot\omega_t$.
Thus, if $H_*(\Omega X;R)$ is a free $R$-module, so that (for simplicity) the 
cross product \seqm{H_*(X;R)\otimes H_*(\Omega X;R)}{\times}{H_*(X\times\Omega X;R)} is an isomorphism, and
the coproduct on an element $b\in H_*(\Omega X;R)$ has the form $\vartriangle_*(b)=\sum_i d_i\otimes e_i$, 
then $\bar\Delta$ satisfies
\begin{align}
(-1)^{|a|}\bar\Delta(a\otimes b)&=
\csum{i}{}{(-1)^{|d_i|}\bracket{a(ev_*(1\otimes d_i))\otimes\phi_*([S^1]\otimes e_i)}}+
\csum{i}{}{\bracket{a(ev_*([S^1]\otimes d_i))\otimes\phi_*(1\otimes e_i)}}
\label{EBV}\\
&=\csum{i}{}{(-1)^{|d_i|}\bracket{a\epsilon(d_i)\otimes\phi_*([S^1]\otimes e_i)}}+
\csum{i}{}{\bracket{a(ev_*([S^1]\otimes d_i))\otimes e_i}}.\notag
\end{align}
where $\epsilon\colon\seqm{H_*(\Omega X;R)}{}{R}$ is the augmentation. To complete this formula one needs to determine 
the maps $\phi_*$ and $ev_*$. This latter map defines the \emph{homology suspension} 
$\sigma\colon\seqm{H_*(\Omega X;R)}{}{H_{*+1}(X;R)}$, $\sigma(a)=ev_*([S^1]\otimes a)$, which satisfies the formula
\begin{equation}
\label{EHS}
\sigma(ab)=\sigma(a)\epsilon(b)+\epsilon(a)\sigma(b)
\end{equation}
for any product $ab\in H_*(\Omega X;R)$ induced by the loop composition multiplication on $\Omega X$. In particular, 
$\sigma$ is zero on decomposable elements. 
If $X$ is an $H$-space, one can derive this formula by observing that the following diagram commutes
\[\diagram
(S^1\times S^1)\times(\Omega X\times\Omega X)\rto^{\ID\times T\times\ID} & 
(S^1\times \Omega X)\times(S^1\times \Omega X)\rto^-{ev\times ev} & X\times X\dto^{\mu}\\
S^1\times (\Omega X\times\Omega X)\rto^-{\ID\times\Omega\mu}\uto^{\vartriangle\times\ID\times\ID} & S^1\times\Omega X\rto^{ev} & X,
\enddiagram\]
and that point-wise multiplication of based loops $\Omega\mu$ on $\Omega X$ is homotopy commutative and homotopic 
to the loop composition multiplication on $\Omega X$ (this is a mapping space analogue of Theorem $5.21$, Chapter III 
in~\cite{Whitehead}). Alternatively, it is a consequence of the Homology Suspension Theorem~(\cite{Whitehead}, Chapter VIII). 
The map $\kappa(a)=\phi_*([S^1]\otimes a)$ is a bit more mysterious. At the very least, when $\mu$ is commutative 
one obtains an analogous commutative diagram for $\phi$ together with a derivation formula $\kappa(ab)=\kappa(a)b+a\kappa(b)$, 
while for the case of compact Lie groups, $\kappa$ is trivial since $H_*(\Omega X)$ is concentrated even degrees. 
We consider the case where $X$ is an Eilenberg-Maclane space $K(R,n)$. These can be taken to be commutative topological 
groups, and we may write $K(G,n-1)=\Omega K(G,n)$ with commutative multiplication induced by the one on $K(R,n)$, which
by the way is homotopic to the loop composition multiplication.

\begin{proposition}
\label{PEMBV}
Let $J$ be the image of the cross product
\seqm{H_*(K(R,n-1);R)\otimes H_*(K(R,n);R)}{\times}{H_*(K(R,n-1)\times K(R,n);R)} 
(which is injective by the K\"unneth formula). 
Suppose the coproduct on $a\in H_*(K(R,n-1);R)$ is in the image of the cross product,
that is, it is of the form $\vartriangle_*(b)=\sum_i d_i\times e_i$. Then with respect 
to the isomorphism $h_*$, the BV-operator is given on $a\times b\in J\subseteq H_*(\mc L K(R,n);R)$ 
by the formula
$$
\Delta(a\times b)
=(-1)^{|a|}\csum{i}{}{\bracket{a(\rho_*([S^1]\otimes d_i))\times e_i}},
$$
where \seqm{\Sigma K(R,n-1)}{\rho}{K(R,n)} is a classifying map for  
$[S^1]\otimes\iota_{n-1}\in H^*(\Sigma K(R,n-1);R)\cong \bar H^*(S^1;R)\otimes \bar H^*(K(R,n-1);R)$,
and $\iota_{n-1}\in \bar H^{n-1}(K(R,n-1);R)$ is the fundamental class.

\end{proposition}

\begin{proof}

Since our map $\seqm{S^1\times K(R,n-1)}{\phi}{K(R,n-1)}$ restricts to the identity on the 
right factor, $\phi^*(\iota_{n-1})=1\otimes\iota_{n-1}$, or in other words, $\phi$ is a classifying map 
of the cohomology class $1\otimes\iota_{n-1}\in\bar H^{n-1}(S^1\times K(R,n-1);R)$. The projection 
map onto the right factor \seqm{S^1\times K(R,n-1)}{\ast\times\ID}{K(R,n-1)} is also a classifying 
map for $1\otimes\iota_{n-1}$. Since cohomology classes are in one-to-one correspondance with the
homotopy classes of the classifying maps representing them, $\phi$ must be homotopic to $\ast\times\ID$.
Therefore $\phi_*([S^1]\otimes d)=0$ for any $d$.

Next, recall the suspension isomorphism \seqm{H_{n-1}(K(R,n-1);R)}{\cong}{H_{n}(\Sigma K(R,n-1);R)}, 
sending $a\mapsto [S^1]\otimes a$, factors as the composite 
$$
\seqmm{H_{n-1}(K(R,n-1);R)}{\cong}{[K(R,n-1),K(R,n-1)]}{\cong}{[\Sigma K(R,n-1),K(R,n)]} 
$$
where the last map is the adjoint isomorphism. Since the evaluation map \seqm{S^1\times K(R,n-1)}{ev}{K(R,n)} 
restricts to the constant map on both the left and right factors, it factors as the composite
$$ 
ev\wcolon\seqmm{S^1\times K(R,n-1)}{quotient}{\Sigma K(R,n-1)}{ev'}{K(R,n)}, 
$$
where the last map $ev'$ (also known as the evaluation map in the literature) is the adjoint of 
the identity map \seqm{K(R,n-1)}{\ID}{K(R,n-1)}. Since the identity is a classifying map of 
$\iota_{n-1}$, by the above factorization of the suspension, its adjoint $ev'$ is a classifying map of 
$[S^1]\otimes\iota_{n-1}$. The proposition now follows using equation~(\ref{EBV}).

\end{proof}

The BV-operator has a very clean form on decomposable elements when we take our multiplication on $H_*(\mc L X)$ 
to be the one induced by point-wise multiplication of loops $\mc L\mu$ (instead of the multiplication 
$(\Omega\mu\times\mu)\circ(\ID\times T\times\ID)$ based on each coordinate of $\Omega X\times X\cong\mc L X$). 
Tamanoi~\cite{Tam06} gave a derivation formula with respect to this product 
$$
\Delta(ab)=\Delta(a)b+(-1)^{|a|}a\Delta(b),
$$
which is a straightforward consequence of the following commutative diagram
\[\diagram
(S^1\times S^1)\times(\mc L X\times \mc L X) \rto^-{\ID\times T\times\ID} & 
(S^1\times\mc L X)\times (S^1\times\mc L X)\rto^-{\nu\times\nu} & 
\mc L X\times\mc L X\dto^{\mc L\mu}\\
S^1\times (\mc L X\times \mc L X)\rto^{\ID\times\mc L\mu}\uto^{\vartriangle\times\ID\times\ID} & 
S^1\times \mc L X\rto^-{\nu} & \mc L X.
\enddiagram\]
Both multiplications on $\mc L X$ are equal when the multiplication on $X$ is commutative. Since this is the case
for $K(R,n)$, our formula in Proposition~\ref{PEMBV} satisfies
\begin{equation}
\label{EBVderivation}
(-1)^{|b||c|}\Delta(ac\times bd)=\Delta((a\times b)(c\times d))
=\Delta(a\times b)(c\times d)+(-1)^{|a|+|b|}(a\times b)\Delta(c\times d)
\end{equation}
The derivation formula can also be used to determine how the BV-operator interacts with the cross-product, 
as we see in the following:

\begin{proposition}
\label{PCPBV}
Let $X=X_1\times\cdots\times X_k$ be a product of topological groups $(X_i,\mu_i)$. 
Then the BV-operator for $\mc L X\cong \mc L X_1\times\cdots\times \mc L X_k$ satisfies 
$$
\Delta(a_1\times\cdots\times a_k)=\csum{i}{}{(-1)^{k_i}\bracket{a_1\times\cdots\times\Delta(a_i)\times\cdots\times a_k}}
$$
for $a_i\in H_*(\mc L X_i)$, where $k_i=\sum_{j=1}^{i-1}|a_{j}|$ and $k_1=0$.
\end{proposition}

\begin{proof}

It suffices to prove the statement for length-$2$ products $X=X_1\times X_2$. One can then iterate to obtain 
the general formula. Since the inclusion of the left factor \seqm{\mc L X_1}{\ID\times\ast}{\mc L X_1\times \mc L X_2} 
induces the map on homology sending $a\mapsto a\times 1$ for any $a$, by naturality of the BV-operator we have
$\Delta(a_1\times 1)=(\ID\times\ast)_*(\Delta(a_1))=\Delta(a_1)\times 1$. Similarly, 
$\Delta(1\times a_2)=1\times\Delta(a_2)$. Since $X$ is a topological group with multiplication $\mu$ defined by 
the composite \seqmm{X\times X}{\ID\times T\times\ID}{(X_1\times X_1)\times (X_2\times X_2)}{\mu_1\times\mu_2}{X},
the point-wise loop multiplication $\mc L\mu$ is the composite
$$
\seqmmm{\mc L X\times \mc L X}{\cong}{(\mc L X_1\times \mc L X_2)\times (\mc L X_1\times\mc L  X_2)}
{\ID\times T\times\ID}{(\mc L X_1\times \mc L X_1)\times (\mc L X_2\times\mc L  X_2)}
{\mc L \mu_1\times\mc L \mu_2}{\mc L X}.
$$
Therefore $(a_1\times 1)(1\times a_2)=a_1\times a_2$ with respect to this induced  product, and by the derivation
formula we have
\begin{align*}
\Delta(a_1\times a_2)&=\Delta(a_1\times 1)(1\times a_2)+(-1)^{|a_1|}(a_1\times 1)\Delta(a_2\times 1)\\
&=\Delta(a_1)\times a_2+(-1)^{|a_1|}a_1\times\Delta(a_2).
\end{align*}

\end{proof}

We have, for the sake of simplicity, been restricting $X$ to be a topological group. Some of the material above
however extends (up-to-homotopy) to where $X$ is a homotopy associative $H$-space. In this scenario $h$ is a 
homotopy equivalence since it defines is a weak equivalence between the free loop fibration and the trivial 
fibration. If $X$ has an inverse $-\ID\colon\seqm{X}{}{X}$, $x\mapsto x^{-1}$, the null homotopy 
$H\colon\seqm{X\times X\times I}{}{X}$, with $H_0=\ast$ and $H_1=\ID\times-\ID$, allows us to define the 
homotopy inverse $h^{-1}$ just as before, except this time composing the loop $\omega(0)^{-1}\cdot\omega$ with 
the based path given by $H_t(\omega(0)^{-1},\omega(0))$, and the action $\bar\nu$ will have a similar form.

In the case of rational coefficients, a simply connected $H$-space $X$ has a rational decomposition 
$X_{\Q}\simeq \prod_i K(\Q,n_i)$, 
and the classifying maps \seqm{\Sigma K(\Q,n_i-1)}{}{K(\Q,n_i)} can be identified with the Freudenthal 
suspension \seqm{S^{n_i}_{\Q}}{}{\Omega\Sigma S^{n_i}_{\Q}} in the $n_i$ even case, and evaluation 
\seqm{\Sigma\Omega S^{n_i}_{\Q}}{}{S^{n_i}_{\Q}} in the odd case. We see then that the action of 
$\Delta$ on $H_*(\mc L X;\Q)$ with respect to the algebra structure induced by the group multiplication on 
$\prod_i K(\Q,n_i)$ can be determined by applying Propositions~\ref{PCPBV} and~\ref{PEMBV}.

This technique can still be used to obtain some useful information for more general coefficients. Suppose 
$H_*(X;R)$ is free as an $R$-module, and $a\in H_n(X;R)$ is an indecomposable element in the Hopf algebra 
$H_*(X;R)$. Then the cohomology dual $\hat a\in H^n(X;R)$ of $a$ is a primitive element in the dual Hopf algebra 
$H^*(X;R)$, the classifying map \seqm{X}{c}{K(R,n)} of $\hat a$ is an $H$-map, and moreover it is natural with 
respect to the homeomorphism $h$. That is, the following squares commute up to homotopy
\begin{equation}
\diagram
X\times X\rto^-{c\times c}\dto_{\mu} & K(R,n)\times K(R,n)\dto^{mult.} & 
X\times\Omega X\rto^-{c\times\Omega c}\dto_{h}^{\cong} & K(R,n)\times\Omega K(R,n)\dto^{h}_{\cong}\\
X\rto^-{c} & K(R,n) &
\mc L X\rto^-{\mc L c} & \mc L K(R,n).
\enddiagram
\end{equation}
The proof of commutativity is as follows. For degree reasons, the fundamental class $\iota_n$ satisfies
$(mult.)^*(\iota_n)=(\iota_n\times 1+1\times\iota_n)$, 
so we have $(c\times c)^*\circ (mult.)^*(\iota_n)=\hat a\otimes 1+1\otimes \hat a$. Likewise, since $\hat a$ is primitive, 
$\mu^*\circ c^*(\iota_n)=\mu^*(\hat a)=\hat a\otimes 1+1\otimes \hat a$. Thus both the composites in the first square are 
classifying maps of $\hat a\otimes 1+1\otimes \hat a$, meaning they are homotopic. This gives the first square. 
To obtain the second square, let $H\colon\seqm{(X\times X)\times I}{}{K(R,n)}$ be a choice of homotopy between 
the composites in the first square. Define the homotopy $G\colon\seqm{(X\times\Omega X)\times I}{}{\mc L K(R,n)}$ by 
$G(x,\omega,t)=\omega_{x,t}$, where $\omega_{x,t}\colon\seqm{S^1}{}{X}$ is the loop given by $\omega_{x,t}(s)=H(x,\omega(s),t)$. 
Then $G$ defines a homotopy between the two composites in the second square. As a consequence of these diagrams, 
$\mc L c_*$ is an algebra map with respect to the algebra structure induced by the isomorphisms $h_*$, given by 
$(\mc L c)_*(v\otimes b)=c_*(v)\times (\Omega c)_*(b)$ .   

Now suppose $n$ is odd, $a$ is trangressive, and $\tau(a)\in H_{n-1}(\Omega X;R)$ is its trangression. Since $c_*$ maps 
$a$ to the homology dual $\hat\iota_n$ of $\iota_n$, and $\hat\iota_n$ is trangressive onto $\tau(\hat\iota_n)=\hat\iota_{n-1}$, 
the homology dual of the fundamental class of $\Omega K(R,n)=K(R,n-1)$, we have $(\Omega c)_*(\tau(a))=\hat\iota_{n-1}$. Then
$
(\mc L c)_*(\Delta(v\otimes\tau(a)))=\Delta((\mc L c)_*(v\otimes\tau(a)))=
\Delta(c_*(v)\times\hat\iota_{n-1})=(-1)^{|v|}(c_*(v)\hat\iota_{n})\times 1
$  
by Proposition~\ref{PEMBV}, and applying the derivation formula~(\ref{EBVderivation}) inductively,
$$
(\mc L c)_*(\Delta(v\otimes\tau(a)^k))=
\Delta(c_*(v)\otimes\hat\iota_{n-1}^k)=k(-1)^{|v|}((c_*(v)\hat\iota_n)\times\hat\iota_{n-1}^{k-1}).
$$ 
Since $(\mc L c)_*(va\otimes\tau(a)^{k-1})=(c_*(v)\hat\iota_n)\times\hat\iota_{n-1}^{k-1}$, if we assume $\tau(a)^{k-1}$ 
generates $H_{(k-1)(n-1)}(\Omega X;R)$, and $va$ generates $H_{n+|v|}(X;R)$, then
$$
\Delta(v\otimes\tau(a)^k)=k(-1)^{|v|}(va\otimes\tau(a)^{k-1}).
$$
For example, if we take $R=\zmodp$ for $p$ an odd prime, $X=S^{n}_{(p)}$ as a $p$-localized sphere 
(which is an $H$-space for $n$ odd~\cite{Adams2}), and $a=[S^{n}]$, then this formula completely determines the action of 
$\Delta$ on $H(\mc L S^{n};\zmodp)\cong H(\mc L X;\zmodp)$. This is a somewhat different approach for spheres than the one 
taken by Westerland in~\cite{Westerland}, and Menichi in~\cite{Menichi1}.

\section{Proof of Theorem~\ref{TMain2}}

For degree placement reasons, it is clear that $\Delta(\mc Q)\subseteq \mc W$, $\Delta(\mc W)\subseteq \mc Z$, and
$\Delta(\mc Z)=\{0\}$ when $n>3$. Consider the composite
$$
f\wcolon\seqmm{M}{\vartriangle}{\prod_{i=1}^m M}{\prod_i f_i}{\prod_{i=1}^m K(\Q,n)=P},
$$
where $f_i$ is the classifying map of the generator $a_i\in H^n(M;\Q)$. Let $\iota_i\in H_n(K(\Q,n);\Q)$ denote the 
homology dual of the fundamental class for the $i^{th}$ factor, and $\bar\iota_i\in H_{n-1}(K(\Q,n-1);\Q)$ the corresponding 
trangression. Let $W=\Q\{\iota_1,\ldots,\iota_m\}$ and $\bar W=\Q\{\bar \iota_1,\ldots,\bar \iota_m\}$.

Since $n$ is odd, $H_{*}(K(\Q,n);\Q)\cong\Lambda_{\Q}[\iota_i]$, $H_{*}(K(\Q,n-1);\Q)\cong\Q[\bar\iota_i]$, $f$ induces the 
injection \seqm{H_*(M;\Q)\cong V\oplus K}{}{\Lambda_{\Q}[W]}, mapping $a_i\mapsto\iota_i$ and 
$[M]\mapsto \beta=\sum_{i<j}(c_{ij}\iota_i\iota_j)$, 
and $\Omega f$ induces the algebra map \seqm{\mc Q}{\eta_q}{\Q[\bar W]\cong S(V)}, mapping $u_i\mapsto\bar\iota_i$. 
     
Consider the morphism of rational homology Serre spectral sequences \seqm{\mc E}{\phi}{E} induced by the map of free loop space 
fibrations 
\[\diagram
\Omega M\rto^{}\dto^{\Omega f} & \mc L M\rto^{ev_1}\dto^{\mc L f} & M\dto^{f}\\
\Omega P\rto^{} & \mc L P\rto^{ev_1} & P. 
\enddiagram\]
The spectral sequence $E$ for the bottom fibration collapses since the total space is a topological group with section. On the 
infinity page 
$$
H_*(\mc L P;\Q)\cong H_*(P;\Q)\otimes H_*(\Omega P;\Q)\cong \cplus{i=0}{m}{E^{\infty}_{ni,*}},
$$
and $\phi^{\infty}$ restricts to the maps \seqm{\mc Q}{\eta_q}{\Q[\bar W]\cong E^{\infty}_{0,*}}, 
\seqm{\mc W}{\eta_w}{W\otimes \Q[\bar W]\cong E^{\infty}_{n,*}}, and 
\seqm{\mc Z}{\eta_z}{\Q\{\beta\}\otimes \Q[\bar W]\subseteq E^{\infty}_{2n,*}} 
(note $W\cong V$, $\Q\{\beta\}\cong K$, and $\Q[\bar W]\cong S(V)$ in the introduction). 

Let $F$ be the filtration of $H_*(\mc L P;\Q)$ associated with the spectral sequence $E$. Notice $E^{\infty}_{n,*}\cong F_{n,n+*}/\Q[\bar W]$, 
and $\Q[\bar W]$ is concentrated in degrees $k(n-1)$, while $\mc W$ is concentrated in degrees $n+k(n-1)$, which are never
equal when $n>3$, so they do not share any nonzero elements in the same degree. Similarly, $E^{\infty}_{2n,*}\cong F_{2n,2n+*}/F_{n,2n+*}$, 
$F_{n,*}\cong \Q[\bar W]\oplus (W\otimes\Q[\bar W])$ is concentrated in degrees $k(n-1)$ and $n+k(n-1)$, and $\mc Z$ is 
concentrated in degrees $2n+k(n-1)$, which are never equal when $n>3$. Therefore, with respect to our isomorphism 
$H_*(\mc L M;\Q)\cong\mc Q\oplus\mc W\oplus\mc Z$, $(\mc L f)_*$ restricts to the maps $\eta_q$, $\eta_w$, and $\eta_z$ on each summand.

The action of $\Delta$ on $H_{*}(\mc L K(\Q,n-1);\Q)$ is given by $\Delta(1\otimes\bar\iota_i^k)=k(\iota_i\otimes\bar\iota_i^{k-1})$
and $\Delta(a\otimes\bar\iota_i)=0$ when $|a|>0$. This follows from Proposition~\ref{PEMBV}, and iterating formula~(\ref{EBVderivation}). 
Alternatively, it follows from~\cite{Westerland,Menichi1}. Now by Proposition~\ref{PCPBV},
$$
\Delta(a\otimes\bar\iota_1^{k_1}\cdots\bar\iota_m^{k_m})=
\csum{i=1}{m}{k_i(a\iota_{i}\otimes\bar\iota_i^{k_1}\cdots\bar\iota_i^{k_i-1}\cdots\bar\iota_m^{k_m})}\subseteq 
W\otimes\Q[\bar W]\cong A\otimes S(V)
$$
for any integers $k_i\geq 0$. Since for any $q\in\mc Q$, we have $\Delta(q)\in\mc W$,
$$
\Delta\circ\eta_q(q)=\Delta\circ(\mc L f)_*(q)=(\mc L f)_*\circ\Delta(q)=\eta_w\circ\Delta(q),
$$
we obtain the formula for the composite \seqmm{\mc Q}{\Delta}{\mc W}{\eta_w}{A\otimes S(V)}. Similarly we obtain the formula for the
composite \seqmm{\mc W}{\Delta}{\mc Z}{\eta_z}{K\otimes S(V)}.

\section{Acknowledgements}
The second author was supported by a Leibniz-Fellowship from Mathematisches-Forschungsinsitut-Oberwolfach and an 
Invitation to the Max-Planck-Institut f\"ur Mathematik in Bonn. Both authors are grateful to the MFO's hospitality 
to let them spend some time together to work on this project, Jie Wu for suggesting the problem to the first author, 
John McCleary and the anonymous referee for their helpful comments and suggestions.

\bibliographystyle{amsplain}
\bibliography{mybibliography}

\end{document}